\documentclass[11pt,reqno]{amsart}

\usepackage[a4paper,centering,left=3cm,right=3.2cm]{geometry}

\usepackage{amsfonts,amscd,amssymb,amsmath,amsthm,mathrsfs,multirow,xcolor,lscape,longtable,pbox,lipsum,afterpage,graphicx,threeparttable,comment}
\definecolor{mybg}{RGB}{255,255,200} 
\usepackage[thinlines]{easytable}

\usepackage[colorlinks,linkcolor=blue,anchorcolor=blue,citecolor=blue,urlcolor=blue,backref=page]{hyperref}
\usepackage{enumerate}
\usepackage{graphicx} 
\usepackage{setspace} 
\setdisplayskipstretch{1.8} 
\usepackage{enumitem} 
\setitemize{itemsep=3pt, parsep=0pt, leftmargin=20pt, label=$-$}
\setenumerate{itemsep=3pt, parsep=3pt, leftmargin=25pt, label=$-$}

\usepackage{shadowtext} 

\usepackage{microtype}

\usepackage{bm}
\usepackage{bbm}
\usepackage{changepage}
\usepackage{mathtools}
\usepackage{pdfpages} 

\usepackage{LibreBodoni} 

\usepackage{OldStandard} 

\usepackage[autostyle, english = american]{csquotes}
\usepackage[english]{babel}
\MakeOuterQuote{"}

\usepackage[all,cmtip]{xy}	
\usepackage{tikz}			
\usetikzlibrary{arrows} 		
\usetikzlibrary{matrix}		
\usetikzlibrary{calc}		
							
\newtheoremstyle{updog}
  {10pt plus 4pt minus 4pt}
  {10pt plus 4pt minus 4pt}
  {\slshape}
  {}
  {\librebodoni}
  {.}
  {5pt}
  {}

\newtheoremstyle{defstyle}
  {10pt plus 4pt minus 4pt}
  {10pt plus 4pt minus 4pt}
  {}
  {}
  {\librebodoni}
  {.}
  {5pt}
  {}

\newtheoremstyle{remarkstyle}
  {10pt plus 4pt minus 4pt}
  {10pt plus 4pt minus 4pt}
  {}
  {}
  {\librebodoni \slshape}
  {.}
  {5pt}
  {}

\theoremstyle{updog}
\newtheorem{theorem}{Theorem}
\newtheorem{proposition}[theorem]{Proposition}
\newtheorem{lemma}[theorem]{Lemma}
\newtheorem{corollary}[theorem]{Corollary}

\theoremstyle{defstyle}
\newtheorem{definition}[theorem]{Definition}

\theoremstyle{remarkstyle}
\newtheorem{remark}[theorem]{Remark}

\let\originalleft\left 
\let\originalright\right
\renewcommand{\left}{\mathopen{}\mathclose\bgroup\originalleft}
\renewcommand{\right}{\aftergroup\egroup\originalright}

\allowdisplaybreaks 

\usepackage{titlesec} 

\DeclareRobustCommand{\begcase}{\begin{cases}}
\DeclareRobustCommand{\endcase}{\end{cases}}
\DeclareRobustCommand{\begmat}{\begin{pmatrix}}
\DeclareRobustCommand{\endmat}{\end{pmatrix}}
\DeclareRobustCommand{\begsmat}{\lp \begin{smallmatrix}}
\DeclareRobustCommand{\endsmat}{\end{smallmatrix} \rp} 

\DeclareRobustCommand{\lp}{\left(}
\DeclareRobustCommand{\rp}{\right)}
\DeclareRobustCommand{\lbr}{\left\lbrace}
\DeclareRobustCommand{\rbr}{\right\rbrace}

\DeclareRobustCommand{\lep}{\bigl(}
\DeclareRobustCommand{\llep}{\Bigl(}
\DeclareRobustCommand{\lllep}{\biggl(}
\DeclareRobustCommand{\llllep}{\Biggl(}
\DeclareRobustCommand{\rip}{\bigr)}
\DeclareRobustCommand{\rrip}{\Bigr)}
\DeclareRobustCommand{\rrrip}{\biggr)}
\DeclareRobustCommand{\rrrrip}{\Biggr)}

\DeclareRobustCommand{\leb}{\bigl[}

\DeclareRobustCommand{\rib}{\bigr]}

\DeclareRobustCommand{\lebr}{\big\{}
\DeclareRobustCommand{\llebr}{\Big\{}
\DeclareRobustCommand{\lllebr}{\bigg\{}

\DeclareRobustCommand{\ribr}{\big\}}
\DeclareRobustCommand{\rribr}{\Big\}}
\DeclareRobustCommand{\rrribr}{\bigg\}}

\DeclareRobustCommand{\lev}{\big|}
\DeclareRobustCommand{\llev}{\Big|}
\DeclareRobustCommand{\lllev}{\bigg|}
\DeclareRobustCommand{\llllev}{\Bigg|}
\DeclareRobustCommand{\riv}{\big|}
\DeclareRobustCommand{\rriv}{\Big|}
\DeclareRobustCommand{\rrriv}{\bigg|}
\DeclareRobustCommand{\rrrriv}{\Bigg|}

\DeclareRobustCommand{\leV}{\big\|}

\DeclareRobustCommand{\riV}{\big\|}

\DeclareRobustCommand{\dd}{\mathrm{d}}

\DeclareRobustCommand{\im}{\mathrm{Im}}
\DeclareRobustCommand{\re}{\mathrm{Re}}
\DeclareRobustCommand{\vol}{\operatorname{Vol}}
\DeclareRobustCommand{\sumprime}{\sideset{}{'}\sum}

\DeclareRobustCommand{\prim}{\mathrm{pr}}

\DeclareRobustCommand{\breh}{\\[0.5em]}
\DeclareRobustCommand{\bro}{\\[0.8em]}
\DeclareRobustCommand{\bre}{\\[1em]}

\DeclareRobustCommand{\smrt}[1]{\smash{\sqrt{#1}}}
\DeclareRobustCommand{\sfr}[2]{\raisebox{1.5pt}{{\scalebox{0.8}{$\tfrac{#1}{#2}$}}}}
\DeclareRobustCommand{\fr}[2]{\raisebox{0.7pt}{{\scalebox{0.8}{$\displaystyle \frac{#1}{#2}$}}}}

\DeclareRobustCommand{\C}{\mathbb{C}}

\DeclareRobustCommand{\hyp}{\mathbb{H}}

\DeclareRobustCommand{\R}{\mathbb{R}}

\DeclareRobustCommand{\Z}{\mathbb{Z}}

\DeclareRobustCommand{\SL}{\mathrm{SL}}

\DeclareRobustCommand{\0}{{\bf 0}}
\DeclareRobustCommand{\one}{{\bf 1}}

\DeclareRobustCommand{\bfu}{{\bf u}}
\DeclareRobustCommand{\bfv}{{\bf v}}
\DeclareRobustCommand{\bfw}{{\bf w}}

\DeclareRobustCommand{\calB}{\mathcal{B}}

\newcommand{\cusp}{\mathfrak{c}}
\newcommand{\epstein}{{\mathrm Z}}

\newcommand{\difform}[1]{\smash{\widetilde{#1}}}
\newcommand{\bital}[1]{\textbf{\textit{#1}}}

\let\subsubsection\paragraph
\newcommand{\periodafter}[1]{#1.}

\titleformat{\section}
  [block]             
  {\bfseries\librebodoni\centering} 
  {\bfseries\thesection.}       
  {4pt}               
  {}                  
  []                  

\titleformat{\subsection}
  [block]
  {\bfseries\librebodoni}
  {\textbf{\thesubsection.}}
  {1ex}
  {}            

\titlespacing*{\subsection}
  {0pt}   
  {3ex}   
  {1ex}   
  
\titleformat{\paragraph}{\bfseries\scshape}{\thesubsubsection.}{1em}{\periodafter}

\makeatletter
\def\@settitle{%
  \begin{center}%
    \baselineskip\p@%
    \librebodoni\bfseries\large 
    \uppercasenonmath\@title 
    \@title 
  \end{center}%
}
\makeatother   

\title[Spectral $L^2$-Bounds on Eisenstein Series]{Spectral $\boldsymbol{L^2}$-Bounds on Eisenstein Series and \\ Diophantine Properties of Quadratic Forms}

\author{Jitendra Bajpai} 
\address{Department of Mathematics, Christian-Albrechts University of Kiel, 24118 Kiel, Germany and Einstein Institute of Mathematics, Edmond J. Safra Campus, The Hebrew University of Jerusalem, Jerusalem 9190401, Israel}
\email{jitendra@math.uni-kiel.de}

\author{Kristian Holm}
\address{Department of Mathematics, Christian-Albrechts University of Kiel, 24118 Kiel, Germany}
\email{holm@math.uni-kiel.de}

\date{\today}

\makeatletter
\def\@tocline #1#2#3#4#5#6#7{%
  \ifnum #1>\c@tocdepth \else
    \par \addpenalty\@secpenalty \addvspace{#2}%
    \begingroup \hyphenpenalty\@M
    \@ifempty{#4}{%
      \@tempdima\csname r@tocindent\number#1\endcsname\relax
    }{%
      \@tempdima#4\relax
    }%
    \parindent\z@ \leftskip #3\relax \advance\leftskip\@tempdima\relax
    \rightskip\@pnumwidth plus4em \parfillskip-\@pnumwidth
    #5\leavevmode\hskip-\@tempdima
      \ifcase #1
       \or\or \hskip 1em \or \hskip 2em \else \hskip 3em \fi%
      #6\nobreak\relax
    \dotfill\hbox to\@pnumwidth{\@tocpagenum{#7}}\par
    \nobreak
    \endgroup
  \fi}
\makeatother

\makeatletter
\def\l@section{\@tocline{1}{2pt}{1pc}{1.5pc}{}}
\def\l@subsection{\@tocline{2}{2pt}{2.5pc}{2pc}{}}
\makeatother
\DisableLigatures{encoding = *, family = *}

\begin{document}
\thispagestyle{empty}

\begin{abstract}
We give $L^2$-bounds with respect to the height of the spectral parameter for Epstein zeta functions $\epstein_Q(s)$ and non-holomorphic Eisenstein series $E_\cusp (s,z)$ attached to the Hecke congruence subgroup $\Gamma_0 (p)$ of prime level $p$. Our bounds are completely explicit in the level and, additionally, constitute an almost sure logarithmic improvement over previously established bounds for the level-one Eisenstein series. Following the approach of Kelmer{\textendash}Kontorovich{\textendash}Lutsko, we combine the approximate functional equation for $\epstein_Q(s)$ with estimates for the number of pairs of lattice vectors at which $Q$ takes nearly equal values. Under a Diophantine hypothesis on $Q$, we sharpen these counting estimates using results of Eskin{\textendash}Margulis{\textendash}Mozes on the distribution of values of quadratic forms of signature $(2,2)$.
\end{abstract}

\maketitle
\renewcommand{\contentsname}{}   

\section{Introduction}
The relationship between quadratic forms and the spectral theory of automorphic forms provides an interesting bridge between classical and modern topics in number theory. For example, in the case of the non-holomorphic Eisenstein series $E_\Gamma(s,z)$ of $\Gamma = \SL_2 (\Z)$ with $z \in \Gamma \backslash \hyp$, this relationship boils down to the fact that the Eisenstein series can be written as 
\begin{align} \label{eq:Eisenstein series}
E_\Gamma(s,z) 
= 
\frac12 \sum_{\bfv\in\Z^2_{\mathrm{pr}}} \frac{\im (z)^s}{|v_1 z + v_2|^{2s}} 
=
\frac12\sum_{\bfv\in\Z^2_{\mathrm{pr}}}Q_z(\bfv)^{-s}
\qquad (\text{for } \re(s) > 1)
\end{align}
where $\bfv = (v_1, v_2)^\intercal$ runs over all primitive lattice points in $\Z^2$ (characterized by $\gcd(v_1,v_2) = 1$). Hence, the underlying quadratic form attached to $z$, and thus to $E_\Gamma (s,z)$, is the binary positive-definite form
\begin{align} \label{eq:quadratic_form_Qz}
Q_z (\bfv) 
&=
\frac{|v_1 z + v_2|^2}{\im(z)} 
=
\frac{1}{\im(z)} \llep v_1^2 |z|^2 + 2 v_1 v_2 \, \re(z) + v_2^2 \rrip.
\end{align}

The value distributions of binary quadratic forms $Q$ are rather subtle in general: the existence of a pair $(\bfv, \bfw) \in \Z^2$ of lattice vectors whose values $Q(\bfv), Q(\bfw)$ are separated by a small (but non-zero) gap corresponds to an almost-isotropic point of the quaternary form $\difform{Q} = Q \oplus (-Q)$ with signature $(2,2)$. This signature is a notorious exceptional case in the quantitative Oppenheim conjecture: indeed, the inherent complications of this class of forms are apparent in the seminal work \cite{EMM98}, and additionally exemplified by the existence of the complementary paper \cite{EMM05} which is explicitly devoted to quadratic forms of signature $(2,2)$. 

Concretely, without any assumptions on the positive-definite binary quadratic form $Q$, a general upper bound on the asymptotic size (as $X \rightarrow \infty$) of the set
\begin{align} \label{def:VQ_XY_set}\indent
V_Q (X, Y)
:=
\llebr (\bfv , \bfw) \in \lep \Z^2 \! \setminus \! \lbr \0 \rbr \rip^{\! 2} \, \rriv \, 
Q(\bfv), Q(\bfw) \leqslant X, \, \,
\lev Q(\bfv)-Q(\bfw) \riv \leqslant Y \rribr
\end{align}
contains an additional logarithmic factor in $X$, compared to ternary and higher-rank quadratic forms (cf. \cite[Thm. 2.2, Thm. 2.3]{EMM98}). As mentioned, later work of Eskin–Margulis–Mozes \cite{EMM05} studies this case in more detail and in fact provides a Diophantine condition on $Q$, under which logarithmic factor can be removed: if the difference form $\difform{Q}$ is not \textit{extremely well approximable by split integral forms} (see Definition \ref{def:EWAS}), then $|V_Q(X,1)|\ll_Q X$, and consequently $|V_Q(X,Y)| \ll_Q X(1+Y)$ for any $X \geqslant 1$ and $Y \geqslant 0$ (see Lemma \ref{lemma:counting_jointly_bounded_values}). 

Returning to the topic of Eisenstein series, we note from the recent work \cite{KKL, KKL-erratum} of Kelmer{\textendash}Kontorovich{\textendash}Lutsko that spectral growth properties of $E_\Gamma (s, z)$ can be highly sensitive to the existence of many small gaps between the values of the quadratic form $Q_z$. Indeed, the cited paper uses this relationship in combination with Eskin{\textendash}Margulis{\textendash}Mozes' general value asymptotics for binary forms to establish $L^2$-bounds on $E_\Gamma (s, z)$ (with respect to the height of the spectral parameter). \\[+1em]
The purpose of this paper is to use the refined counting estimates (valid for a more exclusive class of forms) from \cite{EMM05} to give almost-sure logarithmic improvements on the spectral second-moment bounds of Kelmer–Kontorovich–Lutsko. We also establish corresponding upper bounds on the Eisenstein series attached to the two cusps $\cusp \in \lbr \infty , 0 \rbr$ of the Hecke congruence subgroup
\begin{align*} 
\Gamma_0 (p) 
:=
\lllebr \gamma 
\in \Gamma \, \rriv \, 
\gamma \equiv_p
\begmat
\ast & \ast \\
0 & \ast
\endmat 
\rrribr
\end{align*}
of prime level $p$. In particular, our upper bounds are explicit in the level $p$.

To state our results, let us write $X_0(p)=\Gamma_0(p) \backslash \hyp$ and $X(1)= \Gamma\backslash \hyp$. In order to compare points as the level $p$ varies, we fix a point $\Gamma w \in X(1)$ and consider the preimage $\pi^{-1}(\Gamma w)$ of $\Gamma w$ under the projection $\pi : X_0(p) \rightarrow X(1).$ Our main result is then as follows.

\begin{theorem}\label{thm:main_thm_for_levelp_EWAS}
Let $p$ be a prime number, and let $w \in \hyp$ be any point with the property that the quadratic form $Q_w \oplus (-Q_w)$ is not EWAS (see Definition \ref{def:EWAS}). Then, for every $z \in \pi^{-1}(\Gamma w)$, the following estimate holds with an implied constant depending only on $w,$ but not on $p$ or $z$: 
\begin{align*} 
&\int_{-T}^T \llev E_\infty \lep \tfrac12 + it, z \rip \rriv^2 \, \dd t  
+ p \int_{-T}^T \llev E_0 \lep \tfrac12 + it, z \rip \rriv^2 \, \dd t  
\\[+0.5em] &\hspace{6.3cm} \ll_w 
T \lep p + \log T \rip (\log T)^{4/3}  (\log \log T)^{2/3}.
\end{align*}
Consequently, for each of the two Eisenstein series we have the estimates
\begin{align*}
\int_{-T}^T \llev E_\cusp \lep \tfrac12 + it, z \rip \rriv^2 \, \dd t
\ll_w 
\begcase
T \lep p + \log T \rip (\log T)^{4/3} (\log \log T)^{2/3} &\text{if } \cusp = \infty, \\[+0.3em]
T  \lep 1 + p^{-1} \log T \rip (\log T)^{4/3} (\log \log T)^{2/3} &\text{if } \cusp = 0,
\endcase
\end{align*}
as $T \rightarrow \infty.$ 
\end{theorem}
\begin{remark}\label{remark:almost_every_improvement}
In Proposition \ref{lemma:non_EWAS_condition} we give an explicit condition on the form $Q_w$ (equivalently, the point $w \in \hyp$) which is sufficient for the non-EWAS assumption in Theorem \ref{thm:main_thm_for_levelp_EWAS} to hold. We show, additionally, that this sufficient condition is satisfied for almost every $w \in \hyp.$ 
\end{remark}
We also have the following uniform version of Theorem \ref{thm:main_thm_for_levelp_EWAS}. 
\begin{corollary}\label{cor:uniform_version}
Let $K \subseteq \hyp$ be compact, and assume that 
\begin{align*}
C_K
:= 
\sup_{w \in K} 
\sup_{X \geqslant 1} 
\frac{|V_{Q_w} (X,1)|}{X}
< 
\infty,
\end{align*}
where $V_{Q_w}(X,1) \subseteq \Z^4$ is the set defined in \eqref{def:VQ_XY_set}. Then the bounds of Theorem \ref{thm:main_thm_for_levelp_EWAS} hold uniformly for $w \in K,$ in the sense that the implied constant only depends on $K$ and $C_K$, but not on $p, z,$ or $w$. 
\end{corollary}

\begin{remark}\label{remark:main_Eisenstein_theorem}
For each fixed $z\in\hyp$ for which $\difform{Q}_z$ is not EWAS, our method improves on the logarithmic factor in the $L^2$-estimates of \cite{KKL, KKL-erratum}: As $T \rightarrow \infty,$ we have
\begin{align*}
\int_{-T}^{T}|E_\Gamma(\tfrac12+it,z)|^2\,dt
\ll_z T(\log T)^{7/3}(\log\log T)^{2/3}.
\end{align*}
Moreover, if $K\subset\hyp$ is compact and satisfies the assumption of Corollary 3, then
\begin{align*}
\sup_{z\in K}\int_{-T}^{T}
|E_\Gamma(\tfrac12+it,z)|^2\,dt
\ll_{K,C_K}
T(\log T)^{7/3}(\log\log T)^{2/3}.
\end{align*}
Indeed, these estimates follow from Theorem \ref{thm:main_epstein_thm} below with $\Lambda = \Z^2$, the estimate \eqref{eq:Korobov_bound}, the identity $2\zeta(2s)E_\Gamma(s,z)=Z_{Q_z}(s)$, and Proposition \ref{lemma:non_EWAS_condition}.
\end{remark}

Theorem \ref{thm:main_thm_for_levelp_EWAS} also gives a bound on the frequency of large values of $E_\cusp (s, z)$. In fact, the same conclusion holds for the level-$1$ Eisenstein series $E_\Gamma$ under the hypothesis of Remark \ref{remark:main_Eisenstein_theorem}.

\begin{corollary}\label{rem:large-values}
Fix a prime $p$, let $w$ satisfy the hypotheses of Theorem \ref{thm:main_thm_for_levelp_EWAS}, and let $z\in\pi^{-1}(\Gamma w)$. For either cusp
$c\in\{\infty,0\}$ and every $\varepsilon>0$, we have
\begin{align*} \indent
\frac{1}{T} \cdot 
\mathrm{meas} \llep
\llebr t\in[T,2T] : \lev E_\cusp  \lep \tfrac12+it,z\rip \riv > T^\varepsilon
\rribr \rrip
\ll_{w,p} 
\frac{(\log T)^{7/3}(\log\log T)^{2/3}}{T^{2\varepsilon}}
\end{align*}
for sufficiently large $T$. Consequently, at each fixed admissible point $z$, the proportion of $s \in \sfr12 + i[T,2T]$ for which $|E_\cusp( s , z)|$ exceeds $T^\varepsilon$ tends to zero as $T\to\infty$.
\end{corollary}
\begin{proof}
By Chebyshev's inequality, the left-hand side is bounded by
\begin{align*}
T^{-1-2\varepsilon}  \int_T^{2T}  
\left|E_\cusp  \lep \tfrac12+it,z\rip \right|^2 \, \dd t.
\end{align*}
The claimed estimate then follows from Theorem \ref{thm:main_thm_for_levelp_EWAS}, since $p$ is fixed.
\end{proof}

The value asymptotics proved in \cite{EMM98} and \cite{EMM05} are, to a large extent, stated with implied constants that depend on the quadratic forms. For this reason, a key insight towards obtaining an implied constant in Theorem \ref{thm:main_thm_for_levelp_EWAS} which is independent of the level is to "separate" $p$ from the quadratic form $Q_w$ whose gaps regulate the growth of $E_\cusp (s, z)$ on the critical line. We accomplish this by considering a version of the Epstein zeta function that sums the values of a quadratic form over a \textit{sublattice} $\Lambda \subseteq \Z^2$ of finite index, rather than the full lattice $\Z^2$; that is, we consider
\begin{align*}
\epstein(Q, \Lambda ; s)
:=
\sumprime_{\bfv \in \Lambda} \frac{1}{Q (\bfv)^s} 
\qquad (\text{for } \re(s) > 1).
\end{align*}
Informally speaking, we are thus able to record the level $p$ in the sublattice $\Lambda,$ while keeping the quadratic form $Q_w$ constant. 

Theorem \ref{thm:main_thm_for_levelp_EWAS} is derived from the following $L^2$-bound on $\epstein (Q, \Lambda, \sfr12 + it)$, which may be of independent interest. (Note that the hypothesis concerning the value counts of $Q$ is precisely a manifestation of the Diophantine condition on $Q$ discussed above.)

\begin{theorem}\label{thm:main_epstein_thm}
Let $Q$ be a positive-definite binary quadratic form with $\det A_Q = 1.$ Let $\Lambda \subseteq \Z^2$ be any sublattice of index $m.$ Assume that the number of pairs $(\bfv, \bfw) \in \Z^2$ with bounded values (that is, $Q(\bfv), Q(\bfw) \ll X$) and uniformly bounded gaps (that is, $|Q(\bfv) - Q(\bfw)| \ll 1$) satisfies $\# V_Q (X, 1) \ll_Q X$. Then, for any $T \geqslant 2,$ we have
\begin{align}
\int_{-T}^T \llev \epstein \lep Q, \Lambda; \tfrac12 + it \rip \rriv^2 \, \dd t
\ll_Q
T \lep m + \log T \rip,
\end{align}
where the implied constant is independent of $\Lambda$ and $m.$ 
\end{theorem}

We end this introduction by briefly commenting on the utility and parallels of $L^2$-estimates such as the bound in Theorem \ref{thm:main_thm_for_levelp_EWAS}. Generally, such bounds are useful for controlling weighted spectral integrals (because of Cauchy{\textendash}Schwarz), and they quantify the frequency of large values as the spectral parameter varies. An application of spectral Eisenstein estimates to effective orbit counting is given by Burrin{\textendash}Nevo{\textendash}R\"uhr{\textendash}Weiss~\cite{BNRW20}.
Uniform pointwise bounds also have arithmetic applications, including estimates for class-group $L$-functions~\cite{Nord21}.

There is a related literature on $L^r$-norms obtained by integrating over the modular surface, such as Humphries' fourth-moment estimates for truncated Eisenstein series~\cite{Hum18}. In that setting, the spectral parameter is fixed and the point on the surface varies. In the present paper, the point $z$ is fixed and the integration is over the spectral parameter.

\subsection{Notation and conventions}
Throughout, we think of points in (e.g.) $\R^n$ and $\Z^n$ as column vectors. The symbol $\sideset{}{'}{\textstyle \sum}$ will be used to denote summation with the zero vector $\0$ omitted. We will also consistently use the following notation:
\begin{itemize}
    \item $\Lambda_p$ will denote the lattice $p \Z \times \Z = \begsmat p & 0 \\ 0 & 1 \endsmat \Z^2$. 
    \item $Q$ will always denote a positive-definite binary quadratic form. $Q_z$ will denote the form \eqref{eq:quadratic_form_Qz} associated to $z \in \hyp.$ $\difform{Q}$ will denote the signature-$(2,2)$ quadratic form given by the difference form $Q \oplus (-Q)$ of $Q.$
    \item $\epstein (Q, \Lambda; s)$ will denote the (sublattice) Epstein zeta function (Definition \ref{def:sublattice_epstein}) associated to a binary quadratic form $Q$ and a lattice $\Lambda \subseteq \R^2$; for the standard lattice, we write $\epstein_Q (s) = \epstein (Q, \Z^2 ; s)$.
    \item $\Z_\prim^2$ will denote the set of primitive lattice points $\bfv = (v_1, v_2)^\intercal \in \Z^2$, characterized by $\gcd(v_1, v_2) = 1$.  
\end{itemize}

\section{Preliminary results}
In the subsections below, we recall some relevant facts about quadratic forms, Eisenstein series, and Epstein zeta functions, and establish a number of auxiliary results that will be needed in the subsequent analysis. 
\subsection{Quadratic forms and their value asymptotics}
We will use the symbol $Q$ to denote a positive-definite binary quadratic form, and write $\difform{Q} := Q \oplus (-Q)$ to denote the difference form whose value at $(\bfv, \bfw) \in \Z^4$ is $Q(\bfv) - Q(\bfw)$. Note that this is an example of a \bital{split} (or \bital{metabolic}) quadratic form: its rank $4$ is precisely twice the dimension of a maximal isotropic subspace (e.g., the diagonal embedding of $\R^2$ in $\R^4$). \par 
If $Q$ has Gram matrix $A_Q,$ so that $Q(\bfv) = \bfv^\intercal A_Q \bfv$ holds identically, then the \bital{dual} quadratic form to $Q$ is defined as $Q^\ast (\bfv) = \bfv^\intercal A_Q^{-\intercal} \bfv = \bfv^\intercal A_Q^{-1} \bfv.$ Note that a binary quadratic form $Q$ with determinant $\det A_Q = 1$ is integrally equivalent to its dual form $Q^\ast$.\\[+1em]
The quadratic form $Q_z$ associated to $z \in \hyp$ has the Gram matrix
\begin{align*}
A_z 
=
\frac{1}{\im (z)}
\begmat
|z|^2 & \re(z) \\
\re(z) & 1
\endmat,
\end{align*}
so that $Q_z (\bfv) = \bfv^\intercal A_z \bfv$ holds identically. The dual form of $Q_z$ is then 
\begin{align*}
Q_z^\ast (\bfv) 
=
\bfv^\intercal A_z^{-1} \bfv
=
\frac{1}{\im(z)} \llep v_1^2 - 2 v_1 v_2 \, \re(z) + |z|^2 v_2^2 \rrip.
\end{align*}
We recall the following notion from Eskin{\textendash}Margulis{\textendash}Mozes \cite[Def. 1.2]{EMM05}, specialized to quaternary forms. 
\begin{definition} \label{def:EWAS}
A quadratic form $F$ with signature $(2,2)$ is called \bital{extremely well-approximable by split forms} (to be abbreviated henceforth as \bital{EWAS}) if the following condition holds: For every $A > 0$ there is a split integral quadratic form $F'$ and a real number $K \geqslant 2$ such that 
\begin{align*}
\leV F - K^{-1} F' \riV \leqslant \frac{1}{K^A},
\end{align*}
where $\Vert \cdot \Vert$ is any norm on the space of quaternary quadratic forms (which can be identified with $\R^{10}$). 
\end{definition}
In the case of the split quadratic form $\difform{Q}_z$ associated to $z \in \Gamma \backslash \hyp$, a sufficient condition for this form to not to be EWAS is provided by the following concrete approximation property.
\begin{proposition} \label{lemma:non_EWAS_condition}
Let $z = x + iy \in \Gamma \backslash \hyp$, write $\alpha = |z|^2 = x^2 + y^2$ and $\beta = 2x,$ and consider the quadratic form $Q_z$ defined in \eqref{eq:quadratic_form_Qz}. 
\begin{enumerate}
\item[i)] Suppose that there are constants $\kappa, C > 0$ such that, for every $q \geqslant 1$ and all $a, b \in \Z$, we have
\begin{align} \label{eq:non_EWAS_condition} \qquad \qquad
\max \lllep \lllev \alpha - \frac{a}{q} \rrriv, \lllev \beta - \frac{b}{q} \rrriv \rrrip 
\geqslant 
\frac{C}{q^{\kappa}}.
\end{align}
Then $\difform{Q}_z$ is not EWAS.
\item[ii)] Moreover, almost every $z \in \Gamma \backslash \hyp$ satisfies the condition in i).
\end{enumerate}
\end{proposition}
\begin{proof}
We first prove Claim \textit{i)}. In the following, any occurrences of $a,b,q, K$ will be understood to depend on a given $A > 0.$ \par 
Consider an integral form $ F(\bfv, \bfw) = a v_1^2 + b v_1 v_2 + q v_2^2 + \cdots$ for $\bfv = (v_1, v_2) ^\intercal$ and $\bfw$ in $\Z^2$. Suppose towards a contradiction that $Q_z$ satisfies \eqref{eq:non_EWAS_condition} and that $\difform{Q}_z$ is EWAS, so that the following holds: \textit{for any $A > 0$ there are integers $a,b,q$ and a real number $K \in [2, \infty)$ such that}
\begin{align} \label{eq:assumptions_for_contradictions} \indent
\lllev \frac{\alpha}{y} - 
\frac{a}{K} \rrriv, \,
\lllev \frac{\beta}{y} - 
\frac{b}{K} \rrriv,\,
\lllev \frac{1}{y} - 
\frac{q}{K} \rrriv 
\leqslant \frac{1}{K^A}.
\end{align}
Then clearly $q \neq 0.$ In fact, we must have $q > 0$ because $y, K > 0$:  taking $A$ so large that $1/2^{A} < 1/2y,$ we have $1/K^A \leqslant 1/2y$. Then \eqref{eq:assumptions_for_contradictions} is violated if $q = 0$. This inequality in fact gives $1/2y \leqslant q/K \leqslant 3/2y$, which forces $q > 0.$ \par 
Since $q \geqslant 1,$ we confront the assumptions \eqref{eq:non_EWAS_condition} and \eqref{eq:assumptions_for_contradictions} by bounding (e.g.) $|\alpha - a / q|$ as follows: By multiplying and dividing by $K q^{-1}$ and using the triangle inequality, we see that this difference equals
\begin{align*}
\lllev \frac{\alpha q}{K} - \frac{a}{K} \rrriv \cdot 
\frac{K}{q}
&=
\lllev \alpha \llep\frac{q}{K} - \frac{1}{y}\rrip 
+\frac{\alpha}{y}
-\frac{a}{K} \rrriv \cdot 
\frac{K}{q}
\breh &\leqslant
\lllev \alpha \llep\frac{q}{K} - \frac{1}{y}\rrip \rrriv
\cdot \frac{K}{q}
+\lllev\frac{\alpha}{y}
-\frac{a}{K} \rrriv 
\cdot \frac{K}{q}
\breh &\leqslant
\frac{\alpha}{K^A}
\cdot \frac{K}{q}
+\frac{1}{K^A}
\cdot \frac{K}{q}.
\end{align*}
Repeating this estimate with $(\beta, b)$ in place of $(\alpha, a)$, we conclude that 
\begin{align} \label{eq:alpha_beta_bounce}
\lllev \alpha - \frac{a}{q} \rrriv , \, 
\lllev \beta - \frac{b}{q} \rrriv 
\leqslant 
\frac{\max ( \alpha + 1, |\beta| + 1 )}{q K^{A-1}}
\leqslant
\frac{c_1}{K^{A-1}}
\end{align}
where $c_1 = c_1 (z) > 0$ is a constant. It now follows from \eqref{eq:non_EWAS_condition} and \eqref{eq:alpha_beta_bounce} that
\begin{align*}
\frac{C}{q^\kappa}
\leqslant 
\frac{c_1}{K^{A-1}}.
\end{align*}
Since $K \geqslant 2,$ this shows that $q$ must be unbounded as $A \rightarrow \infty.$ \par 
We are free to assume that $C$ is so small that $c_2 = c_2(z) := c_1 (z) C^{-1} > 1.$ We also let $A$ be so large that $q \geqslant 2$. Then the inequality above is equivalent to $- \kappa \log q \leqslant \log c_2 - (A-1) \log K$, and thus
\begin{align} \label{eq:for_all_A_bound}
\frac{(A-1) \log K}{\log q} 
\leqslant 
\frac{\log c_2}{\log 2} + \kappa =: c_3 (z)
\end{align}
for all $A \gg 1.$ In particular, since the right-hand side of \eqref{eq:for_all_A_bound} is constant in $A,$ this bound implies that
\begin{align*}
\frac{\log K}{\log q} 
\leqslant 
\frac{c_3}{A - 1}.
\end{align*}
After cross-multiplying and exponentiating, and using that $K \geqslant 2$, we conclude that $q/K$ tends to infinity when $A \rightarrow \infty.$ This is incompatible with the third inequality in \eqref{eq:assumptions_for_contradictions} since $y$ is fixed. \\[+1em]
For Claim \textit{ii)} it will be enough to take $\kappa = 2$ and consider the inequalities
\begin{align} \label{eq:bc_preparation} \indent
\lllev \alpha - \frac{a}{q} \rrriv,  \,
\lllev \beta - \frac{b}{q} \rrriv 
< \frac{1}{q^2}
\qquad (\text{for some $a,b \in \Z$}).
\end{align}
By adding suitable multiples of $q$ to $a$ or $b$, if necessary, we are free to assume that $\alpha, \beta \in [0,1]$. The existence of a pair $(a,b) \in \Z^2$ that satisfies \eqref{eq:bc_preparation} now implies that
\begin{align*}
(\alpha, \beta) \in 
M(q) 
:=
\bigcup_{a = 0}^q
\bigcup_{b = 0}^q \,
\lllep \frac{a}{q}-\frac{1}{q^2}, \, 
\frac{a}{q}+\frac{1}{q^2} \rrrip 
\! \times \!
\lllep \frac{b}{q}-\frac{1}{q^2}, \, 
\frac{b}{q}+\frac{1}{q^2} \rrrip.
\end{align*}
Since $\sum_{q=1}^\infty \vol(M(q)) \ll \sum_{q = 1}^\infty 1/q^2 < \infty,$ the Borel{\textendash}Cantelli lemma shows that the set of $(\alpha, \beta) \in [0,1)^2$ for which \eqref{eq:bc_preparation} holds for infinitely many $q$ is a nullset. 

Consequently, almost every pair $(\alpha, \beta) \in \R^2$ only has finitely many approximations of the form \eqref{eq:bc_preparation}. Given such a pair, we can choose the constant $C > 0$ depending on these finitely many approximations and ensure that \eqref{eq:non_EWAS_condition} is satisfied (with $\kappa = 2$).
 
Finally, we can transfer this metrical result to the set of $(\alpha, \beta) = (x^2 + y^2, 2x)$ for $x + iy \in \hyp$ by using the inverse function theorem and the fact that the Jacobian determinant of $(x,y) \mapsto (x^2 + y^2 , 2x)$ is everywhere non-zero for $x + iy \in \hyp.$ The proof is therefore concluded.
\end{proof}

As in \cite{KKL}, our argument for estimating the Epstein zeta function on the critical line will involve counting jointly bounded values of $Q_z$ and $\difform{Q}_z$. To this end, we consider a general form $Q$ and the set \eqref{def:VQ_XY_set}. Whenever $\difform{Q}$ is non-EWAS, we can give a particularly sharp upper bound on the size of $V_{Q}$ with the following lemma. (See also Remark \ref{remark:EWAS_bad_and_non_EWAS_good} after the proof.)
\begin{lemma}\label{lemma:counting_jointly_bounded_values}
Let $Q$ be a positive-definite quadratic form with determinant $1.$ If $\difform{Q}$ is not EWAS, then we have $\# V_Q (X, 1) \ll_Q X$ for any $X \geqslant 1.$ Furthermore, for any $Q$ that satisfies $\# V_Q(X, 1) \ll_Q X$ (for all $X \geqslant 1$), we also have $\# V_Q (X, Y) \ll_Q X (1 + Y)$ for any $Y \geqslant 0$; in particular, this bound holds when $\difform{Q}$ is not EWAS.
\end{lemma}
\begin{proof}
We first note that, since $\difform{Q}$ is split, the form $Q$ can't be a multiple of a rational form: In that case, a suitable multiple $x \difform{Q}$ of $\difform{Q}$ would be a split integral form, and therefore \textit{exactly} approximable by split forms (and accordingly EWAS). \par 
Let $c_Q^{-1}>0$ be the smallest eigenvalue of the Gram matrix $A_Q$ of $Q.$ Then, if a pair $(\bfv, \bfw)$ satisfies $Q(\bfv), Q(\bfw) \leqslant X$, we also have $\Vert \bfv \Vert^2 + \Vert \bfw \Vert^2 \leqslant 2 c_Q X.$ Using this, we apply \cite[Theorem 1.3]{EMM05} to the form $\difform{Q}$: in the notation of \cite{EMM05}, we take $\Omega$ equal to the open unit ball $B_\0 (1)^\circ \subseteq \R^4$ centered at the origin, $T = 2\smrt{ c_Q X}$, and $(a,b) = (-3,3).$ This shows that the number of points in $V_Q(X, 2)$ which are not contained in a rational $\difform{Q}$-isotropic plane is $\ll_Q X$. 

By \cite[Lemma 10.3]{EMM05} and our initial remark, there are at most four such planes. As each such plane $\Pi$ is a $2$-dimensional subspace, and the intersection $\Pi \cap \Z^4$ will be a lattice in $\Pi$ by rationality, each isotropic plane has $\ll_{\difform{Q}} X$ points in common with $V_Q(X, 2)$. Hence, we obtain $\# V_Q (X, 1) \leqslant \# V_Q (X, 2) \ll_Q X$, as claimed.\\[+0.5em]
For the second claim in the lemma, we choose a fixed $X$ and consider the numbers $n_Q (k) = |N_Q (k)|$ for $k = 0, 1, 2, \ldots$ where 
\begin{align*}
N_Q (k) 
:= 
\llebr \bfv \in \Z^2 \, \rriv \, 
0 < Q(\bfv) \leqslant X, \, \,
k \leqslant Q(\bfv) < k+1 \rribr.
\end{align*}
Since any two $\bfv, \bfw \in N_Q (k)$ satisfy $|Q(\bfv) - Q(\bfw)| < 1,$ the assumed bound on $\# V_Q (X,1)$ implies that
\begin{align} \label{eq:Qvalue_count_upper_bound}
\sum_{k \in \Z} n_Q (k)^2 
\leqslant 
\# V_Q (X, 1)
\ll_Q X.
\end{align}
On the other hand, if a pair $(\bfv, \bfw) \in \Z^4$ satisfies $|Q(\bfv) - Q(\bfw) |\leqslant Y,$ then $\bfv \in N_Q (k_1)$ and $\bfw \in N_Q (k_2)$ for non-negative integers $k_1, k_2$ with
\begin{align*} 
\lev k_1 - k_2 \riv 
\leqslant 
Y'
:=
\big\lceil Y \big\rceil + 1.
\end{align*}
Combining these facts and using the Cauchy{\textendash}Schwarz inequality, we deduce that
\begin{align*} 
\# V_Q (X, Y)
&\leqslant 
\sum_{|d| \leqslant Y'}
\sum_{k \geqslant 0} n_Q (k) n_Q (k+d)
\\ &\leqslant 
\sum_{|d| \leqslant Y'}
\lllep \, \sum_{k \geqslant 0} n_Q (k)^2\rrrip^{\!\! 1/2}
\lllep \, \sum_{k \geqslant 0} n_Q (k+d)^2\rrrip^{\!\! 1/2}
\breh &\leqslant 
\sum_{|d| \leqslant Y'}
\sum_{k \geqslant 0} n_Q (k)^2 
\breh &\ll_Q
\lep \big\lceil Y \big\rceil + 1 \rip X
\end{align*}
where we used \eqref{eq:Qvalue_count_upper_bound} in the last step. This completes the proof.
\end{proof}
\begin{remark} \label{remark:EWAS_bad_and_non_EWAS_good}
Note that, without any approximability assumptions on $\difform{Q},$ we only have $\# V_Q (X, Y) \ll_\Omega X(1+Y) \log X$ (where $\Omega$ is a compact set of binary positive-definite forms containing $Q$, cf. \cite[Theorem 5]{KKL}). The absence of a logarithmic factor in the estimate from Lemma \ref{lemma:counting_jointly_bounded_values} illustrates (as a complement to the general discussion in \cite{EMM05}) the significance of the Diophantine properties of $Q$ concerning the size of $V_Q (X, Y)$.
\end{remark}
Later, when estimating the weight functions appearing in the smooth functional equation for the Epstein zeta function, we will frequently encounter sums of the form below. Before discussing the zeta function in more detail, we therefore prove the following estimate which will be useful later on.
\begin{lemma}\label{lemma:estimate_this_maaan}
Let $Y > 0.$ Then for any $A > \sfr12$, we have
\begin{align*}
\sumprime_{\bfv \in \Z^2} \frac{1}{\sqrt{Q(\bfv)}} 
\lllep 1 + \frac{Q(\bfv)}{Y} \rrrip^{\! \! - A}
\ll_{Q, A}
\sqrt{Y}.
\end{align*}
\end{lemma}
\begin{proof}
Let $j \geqslant 0$ be fixed, and consider the values $Q(\bfv) \in (2^j Y, 2^{j+1} Y] = (U, 2U]$. The contribution of these values is
\begin{align*}
&\sum_{\bfv \in \Z^2} \one \! \leb U < Q(\bfv) \leqslant 2U \rib 
\frac{1}{\sqrt{Q(\bfv)}} 
\frac{1}{(1 + Q(\bfv)/Y)^A}
\\ &\qquad \leqslant 
\frac{1}{\sqrt{U}} 
\frac{1}{(1 + U/Y)^A}
\sum_{\bfv \in \Z^2} \one \! \leb U < Q(\bfv) \leqslant 2U \rib 
\ll_Q 
\frac{\sqrt{U}}{(1 + U/Y)^A},
\end{align*}
where we used the fact that $\# \lbr \bfv \in \Z^2 \mid 0 < Q(\bfv) \leqslant X \rbr \ll_Q X.$ We can estimate this upper bound further as
\begin{align*}
\frac{\sqrt{U}}{(1 + U/Y)^A}
=
\frac{\sqrt{2^j Y}}{(1 + 2^j)^A} 
\leqslant
2^{j (1/2 - A)} \sqrt{Y}.
\end{align*}
Summing this upper bound over all $j \geqslant 0$, we obtain a convergent geometric series (as $\sfr12 - A < 0$), and we conclude that the contribution from the values $Q(\bfv) > Y$ is $\ll_{Q,A} \smrt{Y},$ as claimed. \par 
For the values $Q(\bfv) \in (2^{-j-1} Y , 2^{-j} Y ]$ (where $j \geqslant 0$ is fixed), we can bound their contribution by $2^{-j/2} \smrt{Y}$ (up to a constant depending on $Q$). Then, summing this bound over all $j$, we see that the contribution from the values $Q(\bfv) \leqslant Y$ is $\ll_Q \smrt{Y}.$ 
\end{proof}

\subsection{Epstein zeta functions}
A classical analytical object associated to a positive-definite quadratic form $Q$ is its \bital{Epstein zeta function} $\epstein_Q (s)$ (see \cite{epstein-paperI, epstein-paperII}). For a binary form, this function is
\begin{align} \label{def:classical_epstein}
\epstein_{Q}(s) 
=
\epstein(Q, \Z^2 ; s)
=
\sumprime_{\bfv \in \Z^2} \frac{1}{Q (\bfv)^s} \qquad (\text{for } \re(s) > 1).
\end{align}

For the remainder of this subsection, we will let $Q$ denote a positive-definite binary quadratic form (as in the previous subsection). The function $\epstein_Q (s)$ then satisfies the following functional equation (\cite[Thm. 1.4.1]{terras}):
\begin{align} \label{eq:epstein_fcteq}
\pi^{-s} \Gamma(s) \epstein_Q (s) = (\det A_Q )^{-1/2} \pi^{s-1} \Gamma(1-s) \epstein_{Q^\ast} (1-s).
\end{align}

In addition to the identity \eqref{eq:epstein_fcteq}, $\epstein_Q$ also satisfies an approximate functional equation; indeed, several such estimates have been established in the literature: one example is \cite[Thm. 7]{KKL} (which is a specialization of \cite[Thm. 1]{SV05}), which uses sharp truncations of the series \eqref{def:classical_epstein}; and another version with \textit{smooth} truncations, due to Blomer, is provided by the following theorem from \cite[Eq. (2.1)]{blomer_epstein}:
\begin{theorem}[{\normalfont \cite{blomer_epstein}}, $k = 2$] \label{thm:blomer_approx_fcteq}
Let $|t| \geqslant 10.$ If $Q$ is a positive-definite binary quadratic form, and $\Gamma_\R(s) := \pi^{-s/2} \Gamma (\sfr{s}{2})$, then we have
\begin{align*}
\epstein_{Q}\lep \tfrac12 + it \rip
&=
\sumprime_{\bfv \in \Z^2}
\frac{W_t^+ \lep Q (\bfv) \rip }{Q(\bfv)^{1/2 + it}} 
+ 
\frac{1}{\sqrt{\det A_Q}}\frac{\Gamma_\R (1 - 2it)}{\Gamma_\R (1 + 2it)} 
\sumprime_{\bfw \in \Z^2} \frac{W_t^- \lep Q^\ast (\bfw) \rip }{Q^\ast (\bfw)^{1/2 - it}} 
+ O(1),
\end{align*}
where the weight functions $W_t^\pm (y)$ are given by
\begin{align} \label{def:smooth_weights}
W_t^{\pm} (y)
=
\frac{1}{2 \pi i} \int_{(1)} \exp (u^2) \frac{\Gamma(u+ \sfr12 \pm it)}{\Gamma(\sfr12 \pm it)} (\pi y)^{-u} \, u^{-1} \, \dd u.
\end{align}
\end{theorem}
\begin{remark} \label{remark:weight_bounds}
By applying a standard argument using Stirling's approximation and contour-shifting (e.g., see Young's argument in \cite[Eq. (4.3)]{young_levinson}) to the functions in \eqref{def:smooth_weights}, one can prove that the weights $W_t^{\pm}$ satisfy the following bounds, for any $A, y > 0$ and any non-negative integer $j$: If $t \in [\sfr12 U, 3 U ]$ for $U \geqslant 2$, then
\begin{align*} 
U^j \lllev \frac{\partial^j}{\partial t^j} W^{\pm}_t (y)\rrriv
\ll_{A, j} \frac{1}{(1 + y/U)^A}.
\end{align*}
Furthermore, if $|t| \leqslant 2,$ then $W^{\pm}_t (y) \ll_A (1+y)^{-A}.$
\end{remark}
\begin{remark}
One benefit of smooth weights in the approximate functional equation is the potential resulting improvement of the error term. In this regard, note that the admittedly crude error term $O(1)$ can indeed be sharpened (as is also remarked in \cite{blomer_epstein}), and quite substantially so. This requires a more careful handling of the poles of $\epstein_Q (s)$ in the functional equation. We will carry out this argument below from the perspective of the functions $\epstein(Q, \Lambda; s)$ that we introduce next.
\end{remark}
We will be interested in the following \textit{sublattice variant} of $\epstein_Q,$ where only lattice points from a (finite-index) sublattice $\Lambda \subseteq \Z^2$ are included in the sum:  
\begin{align} \label{def:sublattice_epstein}
\epstein(Q, \Lambda ; s)
:=
\sumprime_{\bfv \in \Lambda} \frac{1}{Q (\bfv)^s}.
\end{align}
Note that \eqref{def:sublattice_epstein} is not a new object compared to the classical Epstein zeta function discussed above, as one can express the sublattice variant under consideration in the form of \eqref{def:classical_epstein} by taking the quadratic form $Q' (\bfv) = Q (B \bfv)$ when $\Lambda = B \Z^2$. We elaborate on this in the following lemma.
\begin{lemma} \label{lemma:epstein_identity_sublattice_dual}
Let $\Lambda = B \Z^2 \subseteq \Z^2$ be a sublattice of index $[\Z^2 : \Lambda] = m$. Then the quadratic form $Q' (\bfv) := m^{-1} Q(B \bfv)$ has the same determinant as $Q$; its Epstein zeta function satisfies
\begin{align}
\epstein ( Q', \Z^2 ; s) 
=
m^s \epstein (Q, \Lambda, s),
\end{align}
and when $\det A_Q = 1$, its dual form is $(Q')^\ast = Q' \circ J$ where $J = \begsmat 0 & -1 \\ 1 & 0 \endsmat.$ 
\end{lemma}
\begin{proof}
The Gram matrix of $Q'$ is $m^{-1} B^\intercal A_Q B$ with determinant $m^{-2} (\det B)^2 \det A_Q,$ and the first claim follows. \par 
Next, the Epstein zeta function of $Q'$ is 
\begin{align*}
\epstein (Q', \Z^2; s) 
=
\sumprime_{\bfv \in \Z^2}
\lllep \frac{m}{Q(B \bfv)} \rrrip^s
=
m^s \sumprime_{\bfv \in \Lambda} \frac{1}{Q (\bfv)^s}
=
m^s \epstein (Q, \Lambda, s),
\end{align*}
and we obtain the second claim. \par 
The final claim concerning the dual form of $Q'$ follows from the general fact that the inverse of an invertible symmetric matrix $M \in \SL_2(\R)$ is $J^\intercal M J$.
\end{proof}
The following lemma allows us to keep the quadratic form $Q_z$ fixed, and keep track of the level $p$ in a sublattice of $\Z^2.$
\begin{lemma} \label{lemma:level_in_sublattice}
The map $w \mapsto \epstein(Q_w , \Z^2 ; s)$ is invariant under left-translation by $\Gamma$. Also, for $\gamma \in \Gamma,$ the Epstein zeta function associated to $\Lambda_\gamma := \gamma^\intercal (p \Z \times \Z)$ satisfies
\begin{align*}
\epstein (Q_w, \Lambda_\gamma ; s) 
=
\epstein (Q_{\gamma w}, \Lambda_p; s).
\end{align*} 
\end{lemma}
\begin{proof}
Let $\gamma = \begsmat a & b \\ c & d \endsmat$ and recall the standard relations
\begin{align} \label{eq:impart_transformation_law}
\im (\gamma w) = \frac{\im (w)}{|cw + d|^2},
\qquad 
\begmat \gamma w \\ 1 \endmat
=
\frac{1}{cw + d} \cdot \gamma \begmat w \\ 1 \endmat
\end{align}
for the Möbius action. Taking $z = \gamma w$ in \eqref{eq:quadratic_form_Qz}, we see that $Q_{\gamma w}$ maps $\bfv \in \Z^2$ to 
\begin{align*}
\frac{|\bfv \cdot (\gamma w, 1)^\intercal |^2}{\im( \gamma w)}
=
\frac{1}{|cw+d|^2} \frac{|\gamma^\intercal \bfv \cdot (w, 1)^\intercal|^2}{\im (\gamma w)}
=
\frac{|\gamma^\intercal \bfv \cdot (w, 1)^\intercal|^2}{\im (w)}
=
Q_w (\gamma^\intercal \bfv).
\end{align*}
Since $\gamma^\intercal$ acts as an automorphism of the lattice $\Z^2$, both claims follow.
\end{proof}

From \eqref{eq:epstein_fcteq} and Lemma \ref{lemma:epstein_identity_sublattice_dual}, we obtain the following functional equation for the sublattice Epstein zeta function \eqref{def:sublattice_epstein}: When $\Lambda \subseteq \Z^2$ has index $m,$ we have
\begin{align} \label{eq:sublattice_epstein_fcteq}
\lllep \frac{m}{\pi} \rrrip^{\! \! s} \Gamma(s) \epstein (Q, \Lambda; s)
=
\lllep \frac{m}{\pi} \rrrip^{\!\!  1-s} \Gamma(1-s) \epstein(Q, \Lambda; 1-s).
\end{align}
(The absence of the dual lattice $\Lambda^\ast$ on the right-hand side is explained by the general fact that two-dimensional lattices are always similar to their duals.) 

In addition to the functional equation \eqref{eq:sublattice_epstein_fcteq}, we will need an approximate functional equation for $\epstein( Q, \Lambda; s).$ As announced above, we can modify Theorem \ref{thm:blomer_approx_fcteq} and obtain the following estimate with a sharp error term.

\begin{theorem} \label{thm:sublattice_epstein_approx_fcteq}
Let $Q$ be a positive-definite binary quadratic form with determinant $1$, and let $\Lambda  = M \Z^2$ be a sublattice in $\Z^2$ with index $m$. With the same notation as in Theorem \ref{thm:blomer_approx_fcteq}, for every $t \in \R$ we have
\begin{align*}
\epstein \lep Q, \Lambda; \tfrac12 + it \rip
&=
\sumprime_{\bfv \in \Lambda} 
\frac{W_{t}^+ (Q(\bfv)/m)}{Q(\bfv)^{1/2 + it}} +
\chi_m \lep \tfrac12 + it\rip 
\sumprime_{\bfv \in \Lambda} 
\frac{W_{t}^- (Q(\bfv)/m)}{Q(\bfv)^{1/2 - it}} 
\\ &\qquad + \, 
R_m \lep \tfrac12 + it \rip,
\end{align*}
where, for $s_0 = \sfr12 + it$, 
\begin{align*}
\chi_m \lep \tfrac12 + it \rip
= 
\llep \fr{\pi}{m} \rrip^{\! 2 i t} \frac{\Gamma (\sfr12 - it)}{\Gamma(\sfr12 + it)},
\qquad
R_m \lep s_0 \rip
=
- \frac{\pi^{s_0}}{m^{s_0} \Gamma(s_0)} \lllep \frac{e^{s_0^2}}{s_0} + \frac{e^{(1-s_0)^2}}{1-s_0} \rrrip.
\end{align*}
Moreover, $|\chi_m (\sfr12 + it )| = 1,$ and $|R_m (\sfr12 + it) | \ll m^{-1/2} e^{-t^2 / 2}.$
\end{theorem}
\begin{proof}
Write $\smash{\widetilde{\epstein}}(Q, \Lambda; s)$ for the completed zeta function on the left-hand side of \eqref{eq:sublattice_epstein_fcteq}. This is holomorphic on $\C \setminus \lbr 0, 1 \rbr$ and has poles at $s = 0, 1$ with residues $-1$ and $1$, respectively (cf. \cite[Sect. 3]{terrasbesselseries}). Therefore the integrand in
\begin{align*}
I_m (s) 
:=
\frac{1}{2 \pi i} \int_{(1)} \widetilde{\epstein} \lep Q, \Lambda; s+u \rip
\exp(u^2) \, u^{-1} \, \dd u
\end{align*}
has poles at $u \in \lbr -s, 1-s, 0 \rbr$. If we let $I_m'(s)$ denote the analogue of $I_m (s)$ with the contour $(-1)$ instead, then the functional equation \eqref{eq:sublattice_epstein_fcteq} shows (in combination with the substitution $u \rightsquigarrow -u$) that $I_m'(s) = -I_m (1-s).$ Then, the rapid decay of $\exp(u^2)$ on vertical lines and the residue theorem give
\begin{align*}
I_m (s) + I_m (1-s)
&=
\frac{e^{s^2}}{s} 
+ \frac{e^{(1-s)^2}}{1-s} 
+ \widetilde{\epstein} \lep Q, \Lambda; s \rip.
\end{align*}
We then recover the approximate functional equation by dividing by $(m/\pi)^s \Gamma(s)$ on both sides of this relation and isolating the resulting term $\epstein(Q, \Lambda, \sfr12 + it)$. (Note that for $\re(u) = 1,$ we can write $\epstein(Q, \Lambda, \sfr12 + it + u)$ in terms of its absolutely convergent infinite series, and may therefore change the order of summation and integration.) \par 
Finally, the bound on $R_m (\sfr12 + it)$ now follows from the argument above by estimating the exponential residues for $s = \sfr12 + it$ and using the well-known estimate of the gamma function on the critical line. This completes the proof.
\end{proof}
\subsection[Eisenstein series and cusps of $\Gamma_0 (p)$]{Eisenstein series and cusps of $\,\boldsymbol{\Gamma_0 (p)}$}
Since $\Gamma_0 (p)$ has prime level $p$, it has exactly two inequivalent cusps: $\cusp = 0$ and $\cusp = \infty$ (see \cite[Prop. 2.6]{iwaniectopics}). Their stabilizers with respect to the Möbius action $\Gamma_0 (p) \curvearrowright \hyp$ are given, respectively, by
\begin{align*}
\Gamma_0 = \lbr 
\pm \begin{pmatrix}
1 & 0 \\
p n & 1 
\end{pmatrix} \, \rriv\,\, n \in \Z \rbr,
\quad
\Gamma_\infty =  \lbr 
\pm \begin{pmatrix}
1 & n \\
0 & 1 
\end{pmatrix} \, \rriv \,\, n \in \Z \rbr.
&&
\end{align*}
For $z \in X_0(p)$, the Eisenstein series $E_\cusp (s,z)$ related to a cusp $\cusp$ is 
\begin{align*}
E_\cusp (s,z) 
= 
\sum_{\gamma \in \Gamma_\cusp \backslash \Gamma_0 (p)} 
\im \lep \sigma_\cusp^{-1} \gamma  \cdot z \rip ^s
\qquad (\text{for } \re(s) > 1),
&&
\end{align*}
where $\sigma_\cusp \in \SL(2, \R)$ denotes the \textit{scaling matrix} of $\cusp$. In our cases, we have 
\begin{align*}
\sigma_0 = \begin{pmatrix}
0 & -1/\sqrt{p} \\
\sqrt{p} & 0
\end{pmatrix},
\qquad
\sigma_\infty = \begin{pmatrix}
1 & 0 \\
0 & 1
\end{pmatrix}.
&&
\end{align*}
A straightforward computation now realizes each of the two Eisenstein series (when $\re(s) > 1$) as a sum over a suitable subset of the primitive lattice points: Using \eqref{eq:impart_transformation_law} and the bijections
\begin{align*} 
\Gamma_\infty \backslash \Gamma_0 (p) 
&\longleftrightarrow  
\lebr (pc, d) \in \Z_\mathrm{pr}^2 \ribr / \lbrace \pm 1 \rbrace , \\
\Gamma_0 \backslash \Gamma_0 (p)
&\longleftrightarrow  
\lebr (a,b) \in \Z_\mathrm{pr}^2 : p \nmid a \ribr / \lbrace \pm 1 \rbrace, 
\end{align*} 
we can write 
\begin{align} \label{eq:p_Eisenstein_series}
E_\infty(s,z) 
=
\frac12 \sum_{\substack{\bfv \in \Z^2_\prim \\ p \mid v_1}} \frac{1}{Q_z(\bfv)^s},
\quad
E_0 (s, z)  = \frac{1}{2p^s} \sum_{\substack{\bfv \in \Z_\prim^2 \\ p \nmid v_1}}  \frac{1}{Q_z(\bfv)^s}.
\end{align}
Both functions $E_\cusp (s,z)$ (viewed as functions of the variable $s$) admit meromorphic continuations to all $s \in \C$ and are therefore meaningful outside of the region of convergence of the series \eqref{eq:p_Eisenstein_series}.\par 
We can relate these series to sums over all non-zero lattice points in the form of Epstein zeta functions (attached to certain sublattices): 
\begin{lemma}\label{lemma:clothesline}
Let $\Lambda_p = p \Z \times \Z \subseteq \Z^2.$ Then the Eisenstein series $E_\cusp (s,z)$ of $\Gamma_0 (p)$ satisfy the following identities (understood as equalities of meromorphic functions):
\begin{align}
E_\infty (s,z) \label{eq:clothes_infty}
&=
\frac{\epstein (Q_z, \Lambda_p ; s) - p^{-2s} \epstein(Q_z, \Z^2 ; s)}{2 \zeta(2s)(1 - p^{-2s})},\breh
E_0 (s,z) 
&= \label{eq:clothes_0}
\frac{1}{p^s} 
\frac{\epstein (Q_z, \Z^2 ; s) - \epstein(Q_z, \Lambda_p ; s)}{2 \zeta(2s)(1 - p^{-2s})}.
\end{align}
\end{lemma}
\begin{proof}
Let us write $\epstein_{p} (s,z) := \epstein(Q_z, \Lambda_p ; s)$ and $\epstein (s,z) := \epstein(Q_z, \Z^2 ; s).$ It will be enough to establish \eqref{eq:clothes_infty} and \eqref{eq:clothes_0} for $\re (s) > 1$ because of the identity theorem for holomorphic functions (and the discussion above). \par 
We can write a vector $\bfv \in \Z^2 \setminus \lbr \0 \rbr$ uniquely as $k \bfu$ where $(k, \bfu) \in \Z_+ \times \Z^2_\prim.$ If $p \nmid k,$ then $k \bfu \in \Lambda_p$ implies that $\bfu \in \Lambda_p$; and if $p \mid k,$ there is no restriction on the primitive vector $\bfu.$ Therefore
\begin{align*}
\epstein_p (s,z) 
&=
\sum_{\substack{k \geqslant 1 \\ p \nmid k}} \frac{1}{k^{2s}} 
\sum_{\bfu \in (\Lambda_p)_\prim} \frac{1}{Q_z (\bfu)^s}
+ \sum_{\substack{k \geqslant 1 \\ p \mid k}} \frac{1}{k^{2s}} 
\sum_{\bfu \in \Z_\prim^2} \frac{1}{Q_z (\bfu)^s}
\\ &=
2 \zeta(2s) \lllep 1 - \frac{1}{p^{2s}} \rrrip
E_\infty (s, z) + \frac{\epstein(s,z)}{p^{2s}} .
\end{align*}
This proves the first claim \eqref{eq:clothes_infty}. \par 
To prove the analogous identity \eqref{eq:clothes_0} for $E_0 (s,z),$ notice that we have the relation $Q_z (p v_1, v_2) = p Q_{pz} (v_1, v_2)$ for all $\bfv = (v_1, v_2) \in \Z^2.$ Then, by using that
\begin{align*}
2 \zeta(2s) E_\Gamma (s,z) = \epstein(s, z)
\end{align*}
for the level-one Eisenstein series $E_\Gamma (s,z),$ we can rewrite \eqref{eq:clothes_infty} as
\begin{align} \label{eq:rewrite_clothes_infty}
E_\infty (s, z)
=
\frac{p^s E_\Gamma (s, pz) - E_\Gamma (s,z)}{p^{2s}-1}.
\end{align}
Now, since the scaling matrix $\sigma_0$ normalizes $\Gamma_0 (p)$ and permutes the cusps $\lbr 0, \infty \rbr,$ we have $E_0 (s,z) = E_\infty (s, \sigma_0 z) = E_\infty (s, -1/pz)$; and since $E_\Gamma(s,z)$ is invariant under $z \mapsto S z = -1/z,$ the identity \eqref{eq:rewrite_clothes_infty} implies that
\begin{align} \label{eq:rerewrite_clothes_infty}
E_0 (s,z) 
=
\frac{p^s E_\Gamma (s,z) - E_\Gamma (s, pz)}{p^{2s} - 1}.
\end{align}
Multiplying \eqref{eq:rerewrite_clothes_infty} by $p^s$ and adding \eqref{eq:rewrite_clothes_infty}, we see that $E_\Gamma (s,z) = E_\infty (s,z) + p^s E_0 (s, z)$. Combining this with \eqref{eq:clothes_infty}, we obtain the identity \eqref{eq:clothes_0}.
\end{proof}

\section[\texorpdfstring{$L^2$}{L2}-bounds on the critical line]{$\boldsymbol{L^2}$-bounds on the critical line}
We now use Theorem \ref{thm:sublattice_epstein_approx_fcteq} to establish the promised $L^2$-bound of $\epstein(Q, \Lambda; s)$ on the critical line and prove Theorem \ref{thm:main_epstein_thm}. We follow and modify the argument used in the proof of \cite[Thm. 3]{KKL}.\\[+1em] 
Let us consider the first smooth truncation in the approximate functional equation for $\epstein (Q, \Lambda, \sfr12 + it),$ viz.
\begin{align} \label{def:S_first_half_of_approx_fcteq}
S \lep  Q, \Lambda; \tfrac12 + it \rip 
:=
\sumprime_{\bfv \in \Lambda} 
\frac{W_{t}^+ (Q(\bfv)/m)}{Q(\bfv)^{1/2 + it}}
\end{align}
which appears on the right-hand side of the estimate in Theorem \ref{thm:sublattice_epstein_approx_fcteq}. Then we can express the second term (aside from the factor $\chi_m$) as
\begin{align*}
S\lep Q, \Lambda; \tfrac12 - it \rip 
=
\overline{S\lep Q, \Lambda; \tfrac12 + it \rip},
\end{align*}
since the weights $W_t^\pm$ satisfy $\smash{\overline{W_t^\pm (y)}} = W_{-t}^\pm (y)$ (for all positive $y$). By the triangle inequality and the fact that $|\chi_m(\sfr12 + it)| = 1,$ we then obtain
\begin{align*}
\lev \epstein \lep Q, \Lambda ; \tfrac12 + it \rip \riv
\leqslant 
2 \lev S\lep Q, \Lambda; \tfrac12 - it \rip  \riv + \lev R_m \lep \tfrac12 + it \rip \riv.
\end{align*}
Thus, the inequality $(a+b)^2 \leqslant 2(a^2 + b^2)$ and our bound on the error term from Theorem \ref{thm:sublattice_epstein_approx_fcteq} imply that
\begin{align} \nonumber
\lev \epstein \lep Q, \Lambda ; \tfrac12 + it \rip \riv^2 
&\leqslant 
8 \lev S \lep Q, \Lambda, \tfrac12 + it \rip \riv^2  +
2 \lev R_m \lep \tfrac12 + it \rip \riv^2 
\\ &\ll \label{eq:epstein_critical_line_square_bound}
\lev S \lep Q, \Lambda, \tfrac12 + it \rip \riv^2  +
\frac{1}{m^{1/2} e^{t^2 / 2}}.
\end{align}
Our goal is now to bound the integral of \eqref{eq:epstein_critical_line_square_bound} over the interval $[-T, T]$ (or equivalently, over $[0,T]$ as $\smash{\overline{\epstein(Q, \Lambda; \sfr12 + it)}} = \epstein(Q, \Lambda; \sfr12 - it)$ for any $t \in \R$). To accomplish this, we decompose $[0, T]$ into dyadic intervals and consider different regimes corresponding to the range of $Q(\bfv) / Q(\bfw)$ separately.\\[+1em]
For the dyadic decomposition, we want to cover $[0,T]$ with a sequence of intervals of the form $[U, 2U]$. Let us first take $U \geqslant 2.$ If we fix a non-negative smooth function $f$ with compact support in $(\sfr12 , 3) \subseteq \R,$ with $f(t) \geqslant 1$ for $t \in [1,2],$ we certainly have
\begin{align} \nonumber
\int_U^{2U} \lev S\lep Q, \Lambda ; \tfrac12 + it \rip \riv^2 \, \dd t
&\leqslant 
\int_\R f \lep U^{-1} t \rip \lev S\lep Q, \Lambda ; \tfrac12 + it \rip \riv^2 \, \dd t
\bro &= \nonumber
\int_\R f \lep U^{-1} t \rip 
\llllev \sumprime_{\substack{\bfv \in \Lambda \\ \bfw \in \Lambda}} 
\frac{W_{t}^+ (m^{-1} Q(\bfv))\overline{W_{t}^+ (m^{-1} Q(\bfw))}}
{Q(\bfv)^{1/2 + it} Q(\bfw)^{1/2 - it}} \rrrriv
\, \dd t
\\ &\leqslant \label{eq:dyadic_integral_estimate}
\sumprime_{\substack{\bfv \in \Lambda \\ \bfw \in \Lambda}} 
\frac{|K_{U,f}(Q(\bfv), Q(\bfw))|}{\sqrt{Q(\bfv) Q(\bfw)}}
\end{align}
where we wrote $K_{U,f} (Q(\bfv), Q(\bfw))$ for the expression
\begin{align} \label{def:oscillatory_kernel_thing}
\int_\R f \lep U^{-1} t \rip 
\exp \lllep it \log \frac{Q(\bfw)}{Q(\bfv)} \rrrip
W_{t}^+ (m^{-1} Q(\bfv))\overline{W_{t}^+ (m^{-1} Q(\bfw))} \, \dd t.
\end{align}
This can be estimated with integration by parts and the first weight bound stated in Remark \ref{remark:weight_bounds}: Because of the compact support of $f,$ by integrating $\exp(it \cdots)$ and differentiating the rest of the integrand, we get
\begin{align} \label{eq:osc_kernel_upper_bound}
&\lev K_{U,f} \lep Q(\bfv), Q(\bfw) \rip\riv \nonumber
\\ &\qquad \ll_{A,j} 
U \lllep 1 + U \lllev \log \frac{Q(\bfw)}{Q(\bfv)} \rrriv \rrrip^{\! \! -j}
\! \lllep 1 + \frac{Q(\bfv)}{mU} \rrrip^{\! \! -A}
\! \lllep 1 + \frac{Q(\bfw)}{mU} \rrrip^{\! \! -A}
\end{align}
for any $A, j > 0$. \\[+1em]

For the upper bounds below, we may enlarge the sums of nonnegative terms to include all $\bfv,\bfw\in\Z^2\setminus\{0\}$. This removes the restriction to the sublattice $\Lambda$, while the index $m$ remains explicit in the weights.
\subsection[Estimates depending on the range of $Q(\bfv)/Q(\bfw)$]{Estimates depending on the range of $\boldsymbol{Q(\bfv)/Q(\bfw)}$}
We now analyze the contributions to \eqref{eq:dyadic_integral_estimate} from the different values of $Q(\bfv)/Q(\bfw)$. \\[+1em]
\noindent \underline{\textit{Case 1:} $Q(\bfv)/Q(\bfw) \not \in [\sfr12 , 2 ]$} \\[+0.5em]
Here we have $|\! \log Q(\bfv) / Q(\bfw) | \geqslant \log 2 \gg 1$. If we take $A > \sfr12$ in \eqref{eq:osc_kernel_upper_bound}, we see that the contribution is
\begin{align} \label{eq:case_1_contribution}
&\ll_{Q,A,j} \nonumber
U^{1-j} \sumprime_{\substack{\bfv \in \Lambda \\ \bfw \in \Lambda}} 
\frac{1}{\sqrt{Q(\bfv) Q(\bfw)}}
\lllep 1 + \frac{Q(\bfv)}{mU} \rrrip^{\! \! -A}
\! \lllep 1 + \frac{Q(\bfw)}{mU} \rrrip^{\! \! -A}
\\ &= \nonumber
U^{1-j}
\llllep \sumprime_{\substack{\bfv \in \Lambda}} 
\frac{1}{\sqrt{Q(\bfv)}}
\lllep 1 + \frac{Q(\bfv)}{mU} \rrrip^{\! \! -A} \rrrrip^{\! 2}
\bre &\ll_{Q, A}
m U^{2-j},
\end{align}
where we used Lemma \ref{lemma:estimate_this_maaan} with $Y = m U$ in the last step. \\[+1em]
\underline{\textit{Case 2:} $Q(\bfv)/Q(\bfw) \in [\sfr12 , 2 ]$ for large values of $Q(\bfv), Q(\bfw)$} \\[+0.5em]
In this situation there is $X \geqslant 1$ with $Q(\bfv), Q(\bfw) \asymp X$. Then, with the Taylor expansion of the natural logarithm we see that
\begin{align*}
\lllev \log \frac{Q(\bfw)}{Q(\bfv)} \rrriv 
\asymp 
\frac{Q(\bfv)}{X} \cdot \lllev \frac{Q(\bfw)}{Q(\bfv)} - 1 \rrriv
=
\frac{|Q(\bfw) - Q(\bfv)|}{X}.
\end{align*}
As $\smrt{Q(\bfv) Q(\bfw)} \gg X$, inserting the above into \eqref{eq:osc_kernel_upper_bound} shows that the contribution is 
\begin{align} \label{eq:case_2_contribution}
&\ll_{A, j}
\frac{U}{X}
\lllep 1 + \frac{X}{mU} \rrrip^{\! \! -2A}
\sumprime_{\substack{\bfv \in \Lambda \\ \bfw \in \Lambda}} 
\one_{\calB(X)} \lep Q(\bfv), Q(\bfw) \rip
\lllep 1 + U \cdot \frac{|Q(\bfw) - Q(\bfv)|}{X} \rrrip^{\! \! -j}
\end{align}
where $\calB(X)$ is the set of all $(\alpha , \beta) \in \R^2_+$ with $\max (\alpha , \beta) \in (X, 2X]$ and $\alpha \beta^{-1} \in [\sfr12 , 2]$. The summands above lie in $[1/2^j, 1]$ when $(\bfv , \bfw) \in V_Q (X, X U^{-1}).$ The number of such pairs is, in turn, provided by Lemma \ref{lemma:counting_jointly_bounded_values}, which applies in the current situation (by hypothesis): 
\begin{align*}
\# V_Q \llep X, \tfrac{X}{U} \rrip \ll_Q X \lllep 1 + \frac{X}{U}  \rrrip,
\end{align*}
and thus the contribution to \eqref{eq:case_2_contribution} from the pairs $(\bfv , \bfw) \in V_Q (X, X U^{-1})$ is
\begin{align} \label{case_2_subcontribution_central}
\ll_{Q} (X + U)
\lllep 1 + \frac{X}{m U} \rrrip^{\! \! -2A}.
\end{align}
Now, let $k \geqslant 1$ and consider the pairs $(\bfv, \bfw)$ with
\begin{align*}
\frac{2^{k-1} X}{U} 
<
\lev Q(\bfw ) - Q(\bfv) \riv 
\leqslant 
\frac{2^{k} X}{U}.
\end{align*}
These pairs form a subset of $V_Q (X, 2^k X U^{-1}),$ so by Lemma \ref{lemma:counting_jointly_bounded_values} once again, the contribution of these to \eqref{eq:case_2_contribution} is
\begin{align*}
&\ll_{Q} \nonumber
\frac{U}{X} \lllep 1 + \frac{X}{m U} \rrrip^{\! \! -2A}
\lllep \frac{1}{1 + 2^{k-1}} \rrrip^{\! j} X \lllep 1 + \frac{2^k X}{U} \rrrip
\\ &\ll
\lllep 1 + \frac{X}{m U} \rrrip^{\! \! -2A} \lllep \frac{U}{2^{jk}} + \frac{X}{2^{k (j-1)}} \rrrip.
\end{align*}
Summing the right-hand side over all $k \geqslant 1,$ we obtain the same upper bound as in \eqref{case_2_subcontribution_central} (on the condition that $j \geqslant 2$, to make both geometric series convergent).\par 
In conclusion, the total contribution to \eqref{eq:dyadic_integral_estimate} from the "large" pairs $Q(\bfv), Q(\bfw)$ with $Q(\bfv) \asymp Q(\bfw)$ is 
\begin{align} \label{case_2_total_contribution}
\sum_{\substack{X = \hspace{0.05em} 2^\ell \\ \ell \hspace{0.1em} \geqslant \hspace{0.1em} 0}} 
(X + U ) \lllep 1 + \frac{X}{mU} \rrrip^{\! \! -2A}
\ll_A
U \log (2 m U) + mU
\end{align}
since $A > \sfr12.$ \\[+1em]
\underline{\textit{Case 3:} $Q(\bfv)/Q(\bfw) \in [\sfr12 , 2 ]$ for small values of $Q(\bfv), Q(\bfw)$} \\[+0.5em]
In this case, we have $Q(\bfv), Q(\bfw) \asymp X$ for some $X < 1.$ By \eqref{eq:osc_kernel_upper_bound} and an estimate similar to \eqref{eq:case_2_contribution}, we see that the contribution from these pairs is at most $\ll_Q U.$ 
\subsection{Finishing the proof of Theorem \ref{thm:main_epstein_thm}}
Assembling all the individual contributions from the three cases considered above (namely \eqref{eq:case_1_contribution}, \eqref{case_2_total_contribution}, and the contribution $O_Q (U)$ from \textit{Case 3}), we conclude from \eqref{eq:dyadic_integral_estimate} that 
\begin{align} \label{eq:Sfct_dyadic_L2_bound}
\int_U^{2U} \lev S\lep Q, \Lambda ; \tfrac12 + it \rip \riv^2 \, \dd t
&\ll_{Q} 
U \log (2 m U) + mU.
\end{align} 
Consequently, the estimate \eqref{eq:epstein_critical_line_square_bound} and the exponential decay of the second term (coming from $|R_m (\sfr12 + it)|^2$) give the preliminary $L^2$-bound
\begin{align} \label{eq:epstein_dyadic_L2_bound}
\int_U^{2U} \lev Z\lep Q, \Lambda ; \tfrac12 + it \rip \riv^2 \, \dd t
&\ll_{Q} 
U \log (2 m U) + mU
\qquad (\text{for } U \geqslant 2).
\end{align}
This already gives an $L^2$-bound for the Epstein zeta function over $\sfr12 + i[2, T]$, as we can deduce that 
\begin{align} \label{eq:epstein_almost_full_L2_bound}
\int_2^{T} \lev Z\lep Q, \Lambda ; \tfrac12 + it \rip \riv^2 \, \dd t
&\ll_{Q} 
T \log (2 m T) + mT
\end{align}
by taking $U = 2^\ell$ and summing \eqref{eq:epstein_dyadic_L2_bound} over $1 \leqslant \ell \leqslant \log_2 T$.\par 
To handle the lower positive end of the critical line $(0 \leqslant t \leqslant 2)$, we can settle for a pointwise estimate of the function $S(Q, \Lambda; \sfr12 + it)$. Indeed, the triangle inequality and the second weight bound from Remark \ref{remark:weight_bounds} show that
\begin{align*}
\llev S \lep Q, \Lambda; \tfrac12 + it \rip \rriv 
&\leqslant 
\sumprime_{\bfv \in \Lambda} \frac{|W_t^+ (Q(\bfv)/m)|}{\sqrt{Q(\bfv)}}
\ll_A
\sumprime_{\bfv \in \Lambda} \frac{1}{\sqrt{Q(\bfv)}}
\lllep 1 + \frac{Q(\bfv)}{m} \rrrip^{\! \! -A}.
\end{align*}
Then Lemma \ref{lemma:estimate_this_maaan} bounds the sum on the right-hand side as $\ll_A \smrt{m}.$ Choosing a concrete value of $A > \sfr12$, squaring and integrating the above estimate over $\sfr12 + i[0,2],$ and combining the result with \eqref{eq:epstein_almost_full_L2_bound}, we conclude that
\begin{align} \label{eq:epstein_full_positive_L2_bound}
\int_0^{T} \lev Z\lep Q, \Lambda ; \tfrac12 + it \rip \riv^2 \, \dd t
&\ll_{Q} 
T \log (2 m T) + mT
\ll
T (\log T + m).
\end{align}
\par As we noted above, it is sufficient to establish the $L^2$-estimate over the positive critical line, since $\epstein(Q, \Lambda; s)$ is real-valued for $s \in \R.$ The proof of Theorem \ref{thm:main_epstein_thm} is therefore finished.
\subsection[$L^2$-bounds for $E_{\mathfrak{c}}(s,z)$]{$\boldsymbol{L^2}$-bounds for $\boldsymbol{E_{\mathfrak{c}} (s,z)}$}
We now complete the proof of Theorem \ref{thm:main_thm_for_levelp_EWAS}. 

Let $z \in \pi^{-1}(\Gamma w) \in \Gamma_0 (p) \backslash \mathbb{H}$ be arbitrary. Then we can write $z = \gamma w$ for some $\gamma \in \Gamma.$ By Lemma \ref{lemma:clothesline}, we can write our Eisenstein series $E_\infty (s,z)$ and $E_0 (s,z)$ in terms of sublattice Epstein zeta functions for the quadratic form $Q_z = Q_{\gamma w}$; and by using Lemma \ref{lemma:level_in_sublattice} to move $\gamma$ to the sublattice $\Lambda_p,$ we therefore have
\begin{align} \label{eq:functions_A_and_B}
E_\infty (s,z) 
=
\frac{A(s)-p^{-2s}B(s)}{2 \zeta(2s) D_p (s)},
\qquad 
E_0 (s,z)
=
\frac{1}{p^s} \frac{B(s)-A(s)}{2 \zeta(2s) D_p (s)},
\end{align}
where we wrote $\Lambda_\gamma = \gamma^\intercal (p \mathbb{Z} \times \mathbb{Z})$ (as in Lemma \ref{lemma:level_in_sublattice}) and
\begin{align*} \quad
A(s) := {\mathrm{Z}}(Q_w, \Lambda_\gamma; s),
\qquad 
B(s) := {\mathrm{Z}} (Q_w, \mathbb{Z}^2; s),
\qquad 
D_p(s) := 1-p^{-2s}
\end{align*}
to ease the notation. 

The triangle inequality shows that $|D_p (\sfr12 + it)| \geqslant \sfr12$. Additionally, the estimate $|a+b|^2 \leqslant 2 (|a|^2 + |b|^2)$ allows us to bound the numerators in \eqref{eq:functions_A_and_B}; thus, for an absolute implied constant, we obtain the estimate
\begin{align} \label{eq:Eisenstein_pointwise_bound_on_critical_line}
\nonumber &\big| E_\infty \bigl( \tfrac12 + it,z \bigr) \big|^2 + 
p \big| E_0 \bigl( \tfrac12 + it,z \bigr) \big|^2
\\[+0.8em] &\qquad = \nonumber 
\frac{| A (\sfr12 + it ) - p^{-1-2it} B(\sfr12 + it) |^2 + | B (\sfr12 + it ) -A(\sfr12 + it) |^2}
{4 |\zeta(1 + 2it)|^2 |D_p (\sfr12 + it)|^2}
\\ &\qquad \ll 
\frac{\big| A (\sfr12 + it ) \big|^2 + (1 + p^{-2}) \big| B(\sfr12 + it) \big|^2}
{|\zeta(1 + 2it)|^2}.
\end{align}
Using Korobov's bound \cite{korobov}
\begin{align} \label{eq:Korobov_bound}
\frac{1}{\zeta(1+it)} \ll (\log t)^{2/3} (\log \log t)^{1/3} \qquad 
(\text{for } t \gg 1),
\end{align}
we conclude from \eqref{eq:Eisenstein_pointwise_bound_on_critical_line} that, for any $|t| \leqslant T,$ 
\begin{align} \label{eq:Eisenstein_improved_pointwise_bound}
\nonumber 
&\big| E_\infty \bigl( \tfrac12 + it,z \bigr) \big|^2 + 
p \big| E_0 \bigl( \tfrac12 + it,z \bigr) \big|^2
\\ &\qquad \ll
\llep \big| A (\sfr12 + it ) \big|^2 + (1 + p^{-2}) \big| B(\sfr12 + it) \big|^2 \rrip 
(\log T)^{4/3} ( \log \log T )^{2/3}.
\end{align}
Integrating the estimate \eqref{eq:Eisenstein_improved_pointwise_bound}, we can use Theorem \ref{thm:main_epstein_thm} with $m = 1$ and $m = p$, respectively, to conclude that
\begin{align*} \quad
&\int_{-T}^T \Big| E_\infty \bigl( \tfrac12 + it , z \bigr) \Big|^2 \, \mathrm{d} t +
p \int_{-T}^T \Big| E_0 \bigl( \tfrac12 + it , z \bigr) \Big|^2 \, \mathrm{d} t
\\ &\qquad \ll_w
\Bigl( T \bigl( p + \log T \bigr) + T \bigl( 1 + \log T \bigr) \Bigr)
(\log T)^{4/3} ( \log \log T )^{2/3}
\\[+0.5em] &\qquad \ll
T \bigl( p + \log T \bigr) (\log T)^{4/3} ( \log \log T )^{2/3}.
\end{align*} 
The implied constants are independent of $\gamma$ since this is the case for the implied constant in Theorem \ref{thm:main_epstein_thm}, and therefore independent of $z.$ This completes the proof of Theorem \ref{thm:main_thm_for_levelp_EWAS}.\\[-0.3em]

Finally, we prove the uniform version stated in Corollary \ref{cor:uniform_version}. 
\begin{proof}[Proof of Corollary \ref{cor:uniform_version}]
It is enough to show that the hypotheses imply uniform versions of the estimates in the subsequent argument that establishes Theorem \ref{thm:main_thm_for_levelp_EWAS}. \par 
First of all, the implied constant in Lemma \ref{lemma:estimate_this_maaan} is uniform in $Q_w$ when $Q_w$ ranges over the compact set $\lbr Q_w \mid w \in K \rbr$. Indeed, this constant only depends on $Q$ via the estimate $\# \lbr \bfv \in \Z^2 \mid 0 < Q(\bfv) \leqslant X \rbr \ll_Q X;$ an explicit upper bound on the left-hand side is given by $c_Q X$ (where $c_Q^{-1} > 0$ is the smallest eigenvalue of $A_Q$), and the claimed uniformity follows. \par 
The assumption $C_K < \infty$ clearly shows that the implied constant in the first part of Lemma \ref{lemma:counting_jointly_bounded_values} is uniform in $Q_w$ when $w \in K.$ Accordingly, the same is true for the implied constant in the second part of this lemma. Since all remaining estimates needed for the proof of Theorem \ref{thm:main_epstein_thm} and the following deduction of Theorem \ref{thm:main_thm_for_levelp_EWAS} are absolute, the proof is completed.
\end{proof} 

\section*{Acknowledgements}
\phantomsection
\addcontentsline{toc}{section}{Acknowledgements}
JB would like to thank the Department of Mathematics of  Christian-Albrechts-Universit\"at zu Kiel and the Hebrew University of Jerusalem (HUJI) for their support and excellent working conditions. In addition, JB was supported by the European Research Council (ERC) under the European Union’s Horizon Europe research and innovation programme (grant agreement No. 101163794, GroupHype). KH would like to thank Valentin Blomer, Anders Södergren, and Simon Kristensen for helpful conversations.

\nocite{}
\bibliographystyle{abbrv}
\phantomsection
\addcontentsline{toc}{section}{References}    

\bibliography{bh}

\end{document}